\documentclass[11pt,reqno]{amsart}

\usepackage{amsfonts,amsthm,mathrsfs,amssymb,amscd,enumerate,url,tikz-cd}

\input{xypic}
\xyoption{all}

\usepackage{xcolor}
\colorlet{mdtRed}{red!50!black}
\definecolor{dblue}{rgb}{0,0,.6}
\usepackage[colorlinks]{hyperref}
\hypersetup{linkcolor=blue,citecolor=dblue,filecolor=dullmagenta,urlcolor=mdtRed}

\numberwithin{equation}{section}
\newtheorem{theorem}[equation]{Theorem}
\newtheorem{corollary}[equation]{Corollary}
\newtheorem{lemma}[equation]{Lemma}
\newtheorem{proposition}[equation]{Proposition}

\newtheorem*{theorem*}{Theorem}
\newtheorem*{corollary*}{Corollary}
\newtheorem*{proposition*}{Proposition}

\theoremstyle{remark}
\newtheorem{remark}[equation]{Remark}

\newcommand{\C}{\mathbb{C}}

\newcommand{\e}{\epsilon}
\newcommand{\bq}{\overline q}

\newcommand{\mf}[1]{\mathfrak{#1}}
\newcommand{\ms}[1]{\mathscr{#1}}
\newcommand{\mb}[1]{\mathbb{#1}}
\newcommand{\mc}[1]{\mathcal{#1}}
\renewcommand{\t}[1]{\widetilde{#1}}

\usepackage{marginnote}
\usetikzlibrary{calc}

\begin{document}

\title[Infinitesimal deformations of some quot schemes, II]
{Infinitesimal deformations of some quot schemes, II} 

\author[I. Biswas]{Indranil Biswas} 
	
\address{Department of Mathematics, Shiv Nadar University, NH91, Tehsil
Dadri, Greater Noida, Uttar Pradesh 201314, India}

\email{indranil.biswas@snu.edu.in, indranil29@gmail.com}

\author[C. Gangopadhyay]{Chandranandan Gangopadhyay} 

\address{Department of Mathematics, Indian Institute of Science Education and Research, Pune, 
411008, Maharashtra, India}

\email{chandranandan@iiserpune.ac.in}

\author[R. Sebastian]{Ronnie Sebastian} 

\address{Department of Mathematics, Indian Institute of Technology Bombay, Powai, Mumbai 
400076, Maharashtra, India}

\email{ronnie@iitb.ac.in} 



\begin{abstract}
Let $C$ be an irreducible smooth complex projective curve of genus $g$, with $g_C\,
\geqslant\, 2$. Let $E$ be a vector bundle on $C$ of rank $r$, with $r\,\geqslant
\,2$. Let $\mc Q\,:=\,\mc Q(E,\,d)$ be the 
Quot Scheme parameterizing torsion quotients of $E$ of degree $d$. 
We explicitly describe all deformations of $\mc Q$. 
\end{abstract}

\maketitle

\tableofcontents

\section{Introduction}

Let $C$ be an irreducible smooth projective curve over $\mb C$ of genus $g_C$, with
$g_C\, \geqslant \,2$. 
Let $E$ be a vector bundle on $C$ of rank $r\, \geqslant\, 1$. 
Fix an integer $d\,\geqslant\, 1$. Let $\mc Q\,:=\,\mc Q(E,\,d)$ be the 
Quot Scheme parameterizing torsion quotients of 
$E$ of degree $d$. It is known that $\mc Q$ is a 
smooth projective variety of dimension $rd$. This Quot scheme has 
been studied by various authors, 
see \cite{BGL}, \cite{BDW}, \cite{BFP},
\cite{BR}, \cite{OP}, \cite{motive-nested-quot-torsion}, \cite{nested-quot-torsion}.

The group of holomorphic automorphisms $\text{Aut}(\mc Q)$ of $\mc Q$ is a complex Lie group
whose Lie algebra is $H^0(\mc Q,\,T_{\mc Q})$ \cite[Lemma 1.2.6]{Sern}.
It is known that 
$$H^0(\mc Q({\mathcal O}^{\oplus r}_C,\,d),\,
T_{\mc Q({\mathcal O}^{\oplus r}_C,\,d)})\,=\, \mf{sl}(r, {\mathbb C})\,=\,
H^0(X,\, {\rm End}({\mathcal O}^{\oplus r}_C))/{\mathbb C}
$$
for all $r\, \geqslant\, 2$ \cite{BDH-aut}. From this it follows that the maximal
connected subgroup of $\text{Aut}(\mc Q({\mathcal O}^{\oplus r}_C,d))$ is
${\rm PGL}(r,{\mathbb C})\,=\, \text{Aut}({\mathcal O}^{\oplus r}_C)/{\mathbb C}^*$,
where $\text{Aut}({\mathcal O}^{\oplus r}_C)\,=\, {\rm GL}(r,{\mathbb C})$ is the
group of holomorphic automorphisms of ${\mathcal O}^{\oplus r}_C$.
More generally, if either $E$ is semistable or $r\,\geqslant\, 3$,
then
$$
H^0(\mc Q(E,\,d),\, T_{\mc Q(E,d)})\,=\, H^0(X,\, \text{End}(E))/{\mathbb C}
$$
\cite{G19}, and hence the maximal connected subgroup of $\text{Aut}(\mc Q(E,\,d))$
is $\text{Aut}(E)/{\mathbb C}^*$, where $\text{Aut}(E)$ is the 
group of holomorphic automorphisms of $E$.

Regarding the next cohomology $H^1(\mc Q,\,T_{\mc Q})$, first consider the
case of $r\,=\,1$. In this case, the Quot scheme
$\mc Q(E,\,d)$ is identified with the $d$-th symmetric product 
$C^{(d)}$ of $C$. The infinitesimal deformation space
$H^1(C^{(d)},\, T_{C^{(d)}})$ for $C^{(d)}$ was computed in \cite{Kempf} under
the assumption that $C$ is 
non-hyperelliptic, and it was computed in \cite{Fantechi} when 
$g\,\geqslant\, 3$. In \cite{Laz-cohomology-symm}, the spaces $H^i(C^{(d)},\, T_{C^{(d)}})$
are described for some values of $i$. Remarkably, a complete description of these 
cohomology groups is unknown, to the best of our knowledge. 

In contrast to the above situation, when $r\,\geqslant\, 2$, a complete description of
the cohomology groups
$H^i(\mc Q,\,T_{\mc Q})$ is obtained in \cite[Theorem 9.10]{BGS-deformations-1}. 

When $g_C\,\geqslant\, 2$, in \cite[Theorem 9.11(2)]{BGS-deformations-1}
it is proved that the vector space 
$H^1(\mc Q,\,T_{\mc Q})$ is an extension of $H^1(C,\, T_C)$ by
$H^1(\Sigma,\,\overline{q}_1^*ad(E))$ (see \eqref{f13} and \eqref{bar-q_1} 
for the definition of $\Sigma$ and $\overline{q}_1$). The elements of
$H^1(\mc Q,\,T_{\mc Q})$ correspond to the first order deformations of the variety $\mc Q$. 
Our aim here is to construct explicitly the deformations of $\mc Q$. This is
carried out using the deformations of the curve $C$ and the deformations of 
$\overline{q}_1^*E$. We show that the images of these deformations of $\mc Q$ 
in $H^1(\mc Q,\,T_Q)$ encompass the entire space $H^1(\mc Q,\,T_Q)$. 
We briefly sketch how this is achieved. 

As a motivating case, first recall 
that when $d\,=\,1$, then $\mc Q(E,\,1)\,\cong \,\mb P(E)$. 
Now note that the projective bundle can also be thought 
of as a \emph{relative} Quot scheme associated to the identity 
morphism $C\,\longrightarrow\, C$ and the vector bundle $E$ together with the
constant Hilbert polynomial $1$, that is, $\mb P(E)\,=\,{\rm Quot}_{C/C}(E,\,1)$. 
Analogously, for $d\,\geqslant\, 2$, given a
deformation of the curve $C$ and a deformation of 
$\overline{q}_1^*E$, 
we consider a certain \emph{relative} 
Quot scheme and construct a first order deformation of this relative 
Quot scheme. This in turn naturally induces first order 
deformation of a very large open subset of $\mc Q$, 
and hence of $\mc Q$ itself. Using \cite[Theorem 9.11(2)]{BGS-deformations-1}
we show that all first order deformations of 
$\mc Q$ arise in this manner. This gives an explicit 
description of all the first order deformations of the variety $\mc Q$.

\section{Recollection of some results}

\subsection{Notation}\label{se9.1}

Let $C$ be an irreducible smooth projective curve defined over $\mathbb C$.
Let $E$ be a vector bundle on $C$.
The $d$-th symmetric product of $C$ is denoted by $C^{(d)}$. 
Let $\mc Q$ denote the Quot scheme parametrizing torsion quotients of $E$
of length $d$. 
Let 
\begin{align}
	p_1\,:\,C\times \mc Q\,\longrightarrow\, C,\ \ p_2\,:\,
	C\times \mc Q\,\longrightarrow\, \mc Q,\label{t5}\\
	q_1\,:\, C\times C^{(d)}\,\longrightarrow\, C,\ \ 
	q_2\,:\,C\times C^{(d)}\,\longrightarrow\, C^{(d)}\label{t6}
\end{align}
denote the natural projections.
Recall that there is a universal exact sequence of coherent sheaves on $C\times \mc Q$
\begin{equation}\label{universal-seq-C times Q}
	0 \,\longrightarrow\, \mc A\,\longrightarrow\, p_1^*E\,\longrightarrow\, \mc B
	\,\longrightarrow\, 0\,.
\end{equation}

Let
\begin{equation}\label{f13}
	\Sigma\,\subset\, C\times C^{(d)}
\end{equation}
be the universal divisor; so $\Sigma\,=\, \{(x,\, \alpha)\, \in\,
C\times C^{(d)}\,\, \big\vert\,\, x\, \in\, \alpha\}$. For ease of notation, let
\begin{equation}\label{bar-q_1}
	\bq_1\,\, :\,\, \Sigma\,\longrightarrow\,C
\end{equation}
be the composition of maps
$\Sigma\,\hookrightarrow\, C\times C^{(d)}\,\stackrel{q_1}{\longrightarrow}\,C$,
where $q_1$ is the map in \eqref{t6}. Note that the fibers of $\bq_1$ are identified
with $C^{(d-1)}$. Let 
\begin{equation}\label{hc}
\phi\,\,:\,\,\mc Q\,\longrightarrow\, C^{(d)}.
\end{equation}
denote the Hilbert-Chow morphism; for the definition, see for example \cite[equation (2.3)]{GS}.

Define
\begin{equation}\label{t7}
	\Phi\,\,:=\,\,{\rm Id}_C\times \phi\,\,:\,\,C\times \mc Q\,\,\longrightarrow\,\, C\times C^{(d)}\,.
\end{equation}
Let
\begin{equation}\label{f16}
	\mc Z \,\, \subset\, \,C\times \mc Q
\end{equation}
be the zero scheme of the inclusion map 
${\rm det}(\mc A)\,\hookrightarrow\, \det (p_1^*E)$, 
where $p_1$ is the map in \eqref{t5}. 
The fact that the determinant map is an inclusion
is a consequence of the following elementary fact. 
Let $R$ be a domain and let $R^{\oplus s}\,\longrightarrow\, R^{\oplus s}$
be an inclusion of free $R$-modules. Then the determinant
of this morphism is nonzero. 
From the definition of $\phi$ it follows 
immediately that $\Phi^*\Sigma\,=\,\mc Z$. 
In fact, $\mc Z$ sits in the following commutative diagram,
in which the middle and right squares are Cartesian,
the curved arrows are the composition of maps and $\overline{p}_1$ is defined
so that the left square commutes:
\begin{equation}\label{dpi}
	\begin{tikzcd}
		C\ar[d,equal] & \mc Z\ar[r]\ar[d]\arrow[rr, bend left=15, "\overline{p}_2", labels=above right]
		\arrow[l, "\overline{p}_1", labels=above] & 
		C\times \mc Q\ar[r, "p_2", labels=below]\ar[d, "\Phi"] & \mc Q\ar[d, "\phi"] \\
		C & \Sigma \ar[r]\arrow[rr, bend right=15, "\overline{q}_2", labels=below right]
		\arrow[l, "\overline{q}_1", labels=above] & 
		C\times C^{(d)}\ar[r, "q_2", labels=above] & C^{(d)}
	\end{tikzcd}
\end{equation}
The above map $\overline{q}_2$ is a finite morphism. Thus, it follows that the 
map $\overline{p}_2$, being the Cartesian product of $\phi$ and 
$\overline{q}_2$, is also a finite map. 

\subsection{The results}
In this section we recall the results related to the computation of cohomology of the tangent bundle $T_{\mc Q}$.
The following proposition helps us relate the cohomology of the tangent bundle with cohomology of the sheaf $\ms Hom(\mc A, \, \mc B)$ on $C\times \mc Q$. 

\begin{proposition}[{\cite[Proposition 9.8]{BGS-deformations-1}}]\label{T_Q}
	The tangent bundle of $\mc Q$ is $$T_{\mc Q}
	\,\,\cong\,\, p_{2*}(\ms Hom(\mc A,\,\mc B)),$$
	where $p_2$ is the projection in \eqref{t5}.
\end{proposition}

To compute the cohomology of the sheaf $\ms Hom(\mc A, \, \mc B)$, we compute its direct images under the map $\Phi$ and use the Leray spectral sequence. 
In order to compute the direct images of $\ms Hom(\mc A,\, \mc B)$ under $\Phi$, we need to compute direct images of various other sheaves on $C\times \mc Q$.
We briefly recall some of these results 
from \cite{BGS-deformations-1} that will be needed. 

\begin{corollary}[{\cite[Corollary 9.1]{BGS-deformations-1}}]\label{cor-direct image of structure sheaf}
	The following statements hold:
	\begin{enumerate}
		\item $\Phi_*\mc O_{C\times \mc Q}\,=\, \mc O_{C\times C^{(d)}}$\ and\ 
		$R^i\Phi_*\mc O_{C\times \mc Q}\,=\,0\,\ \ \text{ for\ all }\ i\,>\,0$.
		
		\item $\phi_*\mc O_{\mc Q}\,=\,\mc O_{C^{(d)}}$\ and\ $R^i\phi_*\mc O_{\mc Q}
		\,=\,0\,\ \ \text{ for\ all }\ i\,>\,0$.
	\end{enumerate}
\end{corollary}

\begin{corollary}[{\cite[Corollary 9.3]{BGS-deformations-1}}]\label{Phi_*B}
	The natural map 
	$$q_1^*E\big\vert_\Sigma\,\longrightarrow\, \Phi_*\mc B$$
	is an isomorphism, where $\Sigma$ is defined in \eqref{f13}
	and $q_1$ (respectively, $\Phi$) is the map in \eqref{t6}
	(respectively, \eqref{t7}). Moreover, $$R^i\Phi_*\mc B\,=\,0$$
	for all $i\,>\,0$.
\end{corollary}
The following theorem reduces the computation of higher direct images of $\ms Hom(\mc A,\, \mc B)$ to the vector bundle $\ms End(\mc A)$. 
\begin{theorem}[{\cite[Theorem 9.5]{BGS-deformations-1}}]\label{preliminiary ses}
	There is a map $\Xi$ that fits in a short exact sequence
	\begin{equation*}
		0\,\longrightarrow\, ad(q_1^*E\big\vert_{\Sigma})\,
		\stackrel{\Xi}{\longrightarrow}\, \Phi_*\ms Hom(\mc A,\,\mc B)
		\,\longrightarrow\, R^1\Phi_*\ms End(\mc A)\, \longrightarrow\, 0
	\end{equation*}
	on $\Sigma$.
	For every $i\,\geqslant\, 1$, there is a natural isomorphism
	\begin{equation*}
		R^i\Phi_*\ms Hom(\mc A,\,\mc B)\,\,\stackrel{\sim}{\longrightarrow}\,\,
		R^{i+1}\Phi_*\ms End(\mc A).
	\end{equation*}
\end{theorem}
In the proof of Theorem \ref{preliminiary ses}, the existence of the 
following two commutative diagrams \cite[equations (9.12), (9.13)]{BGS-deformations-1} are established:
\begin{equation}\label{preliminary ses e2}
	\begin{tikzcd}
		0 \ar[r] & \mc O_{\mc Z} \ar[r] \ar[d,"\cong"] & {\ms End}\left(p_1^*E\big\vert_{\mc Z}\right) \ar[r] \ar[d] & 
		{ad}\left(p_1^*E\big\vert_{\mc Z}\right)
		\ar[r] \ar[d] & 0 \\
		0 \ar[r] & {\ms Hom}(\mc B,\, \mc B) \ar[r] & {\ms Hom}(p_1^*E,\,\mc B)
		\ar[r] & {\ms Hom}(\mc A,\, \mc B) &
	\end{tikzcd}
\end{equation}
The rows in \eqref{preliminary ses e2} are exact. 
Applying $\Phi_*$ to \eqref{preliminary ses e2} and using Corollary \ref{Phi_*B} we get the
following diagram:
\begin{equation}\label{preliminary ses e3}
	\xymatrix{
		0 \ar[r] & \mc O_{\Sigma} \ar[r] \ar[d]^{\cong} & {\ms End}\left(q_1^*E\big\vert_{\Sigma}\right)
		\ar[r] \ar[d]^{\cong} & {ad}\left(q_1^*E\big\vert_{\Sigma}\right)
		\ar[r] \ar[d]^{\Xi} & 0 \\
		0 \ar[r] & \Phi_*{\ms Hom}(\mc B, \,\mc B) \ar[r] & \Phi_*{\ms Hom}(p_1^*E,\,\mc B)
		\ar[r] & \Phi_*{\ms Hom}(\mc A, \,\mc B) &
	}
\end{equation}
(the right vertical arrow is defined to be $\Xi$).

Since the map $\Phi$ is flat and $\ms End(\mc A)$ is a vector bundle, we can apply cohomology and base change theorems, and reduce the computation of its direct images to computations of cohomology of certain sheaves on the fiber of the map $\Phi$. Combining this computation with Theorem \ref{preliminiary ses} we get the direct images of $\ms Hom(\mc A,\, \mc B)$. 
\begin{corollary}[{\cite[Corollary 9.7]{BGS-deformations-1}}]\label{corollary main theorems}\mbox{}
	\begin{enumerate}
		\item There is the following short exact sequence on $\Sigma$
		\begin{equation*}
			0\,\longrightarrow\, ad(q_1^*E\big\vert_{\Sigma})\,\stackrel{\Xi}{\longrightarrow} \,
			\Phi_*\ms Hom(\mc A,\,\mc B)\,\longrightarrow\, q_1^*T_C\big\vert_{\Sigma}\longrightarrow 0
		\end{equation*}
		($\Xi$ is the map in \eqref{preliminary ses e3}).

		\item $R^i\Phi_*\ms Hom(\mc A,\,\mc B)\,=\,0$ for all $i\,\geqslant\, 1$.
		
		\item $H^i(\Sigma,\, \Phi_*\ms Hom(\mc A,\,\mc B))\,\stackrel{\sim}{\longrightarrow}
		\,H^i(\mc Z,\,\ms Hom(\mc A,\,\mc B))$ for all $i$.
	\end{enumerate}
\end{corollary}

Now applying the Leray spectral sequence for the sheaf $\ms Hom(\mc A,\, \mc B)$ and the map $\Phi$ we can compute the cohomology of $\ms Hom(\mc A,\, \mc B)$. Combining this with Proposition \ref{T_Q} we get

\begin{theorem}[{\cite[Theorem 9.11]{BGS-deformations-1}}]\label{cor-cohomology of T_Q-1}
	Let $d,\,g_C\,\geqslant\, 2$. Then there is an exact sequence
	$$0 \,\longrightarrow\, H^1\left(\Sigma,\,q_1^*ad(E)\big\vert_{\Sigma}\right)\,\longrightarrow\,
	H^1(\mc Q,\,T_{\mc Q})\,\longrightarrow\, H^1(C,\,T_{C})\,\longrightarrow\, 0\,.$$
\end{theorem}

\section{Explicit description of deformations}\label{section explicit description}

For a $\C$-scheme $X$ we shall denote by $X[\e]$ the scheme $X\times {\rm Spec}\,\C[\e]$.
For a coherent sheaf $\mc F$ on $X$ we shall denote by 
$\mc F[\e]$ its pullback to $X[\e]$ along the natural projection $X[\e]\,\longrightarrow\, X$.
Similarly, for a homomorphism of sheaves $f\,:\,\mc F\,\longrightarrow\, \mc G$ on $X$,
denote by $f[\e]$ the pullback of $f$ to $X[\e]$ by the natural projection.

\subsection{Some deformation functors}

We refer to \cite{Nitsure}, \cite{Sern} and
\cite[\href{https://stacks.math.columbia.edu/tag/0DVK}{Chapter 0DVK}]{Stk}
for basic results on deformation theory. 
Let ${\bf Sets}$ denote the category of sets.
Let ${\bf Art}_\C$ be the category
of all Artin local $\C$-algebras, with residue 
field $\C$. 

Recall from Basic example 4 in \cite[\S~1]{Nitsure} that for a scheme $X$ 
of finite type over $\C$, we have a deformation functor 
${\bf Def}_X$. In particular, we get the deformation functors 
${\bf Def}_C$ and ${\bf Def}_{\mc Q}$. 
Recall from Basic example 2 in \cite[\S~1]{Nitsure}
the deformation functor $\mc D_{\bq_1^*E}$, where 
$\bq_1$ is the map in \eqref{bar-q_1}. 
We begin by describing some of 
the other functors that will be needed. 

We want to define the 
functor ${\bf Def}_{(C,\bq_1^*E)}$.
For $A\,\in\, {\bf Art}_\C$ consider a pair 
$$(\xi\,:\,\mc C\,\longrightarrow\, {\rm Spec}\,A,\,i)\,,$$
where $\xi$ is flat and the pair fits into a Cartesian square 
\begin{equation}\label{id}
\xymatrix{
	C\ar[r]^i\ar[d] & \mc C\ar[d]^\xi\\
	{\rm Spec}\,\C \ar[r] & {\rm Spec}\,A\,.
}
\end{equation}
Let ${\rm Sym}(j)$ denote the group of permutations of
$\{1,\, \cdots,\, j\}$. For $j\,\geqslant\, 0$ define 
$$\mc C^{(j)}\,\,:=\,\,(\overbrace{\mc C\times_{{\rm Spec}\,A}\mc C\times_{{\rm Spec}\,A}
			\cdots\times_{{\rm Spec}\,A}\mc C)}^{j\text{-times}}/{\rm Sym}(j)\,.$$
Let ${\rm Spec}(B)\,\subset\, \mc C$ be an affine open subset such that 
$B$ is flat over $A$. Then $B^{\otimes n}:=B\otimes_AB\otimes_A\ldots\otimes_A B$
is flat over $A$. We have a splitting of the inclusion map
$(B^{\otimes n})^{{\rm Sym}(n)}\subset B^{\otimes n}$ 
given by $\frac{1}{n!}\sum_{\sigma \in {\rm Sym}(n)}\sigma$.
Thus, $(B^{\otimes n})^{{\rm Sym}(n)}$ is flat over $A$, 
being a summand of a flat $A$-module. 
This shows that the map $\xi$ gives rise to flat maps 
$\mc C^{(d)}\,\longrightarrow\, {\rm Spec}\,A$
and $\mc C\times_{{\rm Spec}\,A} \mc C^{(d-1)}\,\longrightarrow\, {\rm Spec}\,A$.
Now define
\begin{equation}\label{u4}
\t\Sigma\,\,:=\,\,\mc C\times_{{\rm Spec}\,A} \mc C^{(d-1)}
\end{equation}
and denote its structure morphism by 
$$\xi_\Sigma\,:\,\t \Sigma\,\longrightarrow\, {\rm Spec}\,A\,.$$
The map $i$ in \eqref{id} gives rise to a closed immersion 
$$i_\Sigma\,:\,\Sigma\,\longrightarrow\, \t \Sigma$$ 
which fits in the Cartesian square
\[\xymatrix{
	\Sigma \ar[r]^{i_\Sigma} \ar[d] & \widetilde{\Sigma}\ar[d]^{\xi_\Sigma}\\
	{\rm Spec}\,\C \ar[r] & {\rm Spec}\,A\,.
}
\]
Consider quadruples of the form $(\xi,\,i,\,\mc E,\,\theta)$, where $\xi$ and $i$ are as above,
$\mc E$ is a locally free sheaf on $\t \Sigma$ and 
$$\theta\,\,:\,\,i_\Sigma^*\mc E\,\stackrel{\sim}{\longrightarrow}\, \bq_1^*E$$
is an isomorphism. 

If $\xi$ and $\xi'$ are two maps as above, then
an $A$-isomorphism $f\,:\,\mc C\,\longrightarrow\, \mc C'$ induces an $A$-isomorphism 
$f_\Sigma\,:\,\t \Sigma \,\longrightarrow\, \t \Sigma'$.
Two such quadruples $(\xi,\,i,\,\mc E,\,\theta)$ and $(\xi',\,i',\,\mc E',\,\theta')$
are said to be equivalent if there is a pair $(f,\, \alpha)$ such that
\begin{itemize}
\item $f\,:\,\mc C\,\longrightarrow \,\mc C'$ is an $A$-isomorphism satisfying the condition 
$f\circ i\,=\,i'$, and

\item $\alpha\,:\,f_\Sigma^*\mc E'\,\longrightarrow\, \mc E$ is an isomorphism for which
$\theta\circ i_\Sigma^*\alpha\,=\,\theta'$ invoking the 
canonical isomorphism 
$$i_\Sigma^*f_\Sigma^*\mc E'\,\cong\, (f_\Sigma\circ i_\Sigma)^*\mc E'\,=\, (i'_\Sigma)^*\mc E'\,.$$
\end{itemize}
Define ${\bf Def}_{(C,\bq_1^*E)}\,:\, {\bf Art}_\C\,\longrightarrow\, {\bf Sets}$ by letting 
${\bf Def}_{(C,\bq_1^*E)}(A)$ to be the set of equivalence classes 
of quadruples $(\xi,\,i,\,\mc E,\,\theta)$ described above. One easily checks that 
this defines a functor which is in fact a deformation functor (see \cite[\S~1]{Nitsure}
for definition of deformation functor) and it satisfies the deformation
condition ({\bf H$\e$}); see \cite[\S2.3]{Nitsure}.

Clearly, there are natural transformations 
$$\mc D_{\bq_1^*E}\,\,\longrightarrow\,\, {\bf Def}_{(C,\bq_1^*E)}
\,\,\longrightarrow\,\, {\bf Def}_C\,.$$

\begin{lemma}\label{ses deformation functors}
There is a natural short exact sequence of vector spaces 
$$
0\,\longrightarrow\, \mc D_{\bq_1^*E}(\C[\e])\,\longrightarrow\,
{\bf Def}_{(C,\bq_1^*E)}(\C[\e])\,\longrightarrow\, {\bf Def}_C(\C[\e])\,\longrightarrow\, 0\,.
$$
\end{lemma}

\begin{proof}
We only need to prove surjectivity on the right, 
because it is clear that there is a short exact sequence
$$0\,\longrightarrow\, \mc D_{\bq_1^*E}(\C[\e])\,\longrightarrow \,
{\bf Def}_{(C,\bq_1^*E)}(\C[\e])\,\longrightarrow\, {\bf Def}_C(\C[\e])\,.$$

First note that for a diagram
\[\xymatrix{
C\ar[r]^i\ar[d] & \mc C\ar[d]^\xi\\
{\rm Spec}\,\C \ar[r] & {\rm Spec}\,\C[\e]
}
\]
representing an equivalence class in ${\bf Def}_C(\C[\e])$,
there is a locally free sheaf $\mc F$ on $\mc C$ 
and an isomorphism $\gamma\,:\,i^*\mc F\,\stackrel{\sim}{\longrightarrow}\, E$.
For the benefit of the reader we sketch a proof (see also \cite[Proposition 4.3]{Chen}). 
Let $U_1\subset C$ be an affine open such that $E$
is trivial over $U$. Then $C\setminus U_1$ is a finite set 
of points. Thus, we may find an affine open set $U_2$ 
such that $(C\setminus U_1)\subset U_2$ and $E$ 
is trivial over $U_2$. Let $U_{12}$ be the affine 
open set $U_1\cap U_2$ and let $R\,:=\,\Gamma(U_{12},\mc O_C)$. 
The bundle $E$ is obtained by gluing trivial bundles over ${\rm Spec}(R)=U_{12}$
using a matrix $T\in {\rm GL}_r(R)$. Let $\widetilde T\,\in \,{\rm GL}_r(R[\e])$ 
be a lift of $T$. The deformation $\mc C$ is obtained by 
gluing $U_1[\e]$ and $U_2[e]$ along $U_{12}[\e]$. This gluing 
corresponds to an isomorphism $\varphi:R[\e]\longrightarrow R[\e]$,
which lifts the identity map $R\,\longrightarrow \,R$. It is clear that 
the map $\varphi^{\oplus r}\circ \widetilde T$ 
$$R[\e]^{\oplus r}\,\,\stackrel{\widetilde T}{\longrightarrow}\,\,R[\e]^{\oplus r} 
\,\,\stackrel{\varphi^{\oplus r}}{\longrightarrow}\,\,R[\e]^{\oplus r}$$
is an isomorphism of modules which is compatible with the map $\varphi$. 
Moreover, this reduces to the map $T$ when we set $\e\,=\,0$. Thus,
using the isomorphism $\varphi^{\oplus r}\circ \widetilde T$ 
we may lift the automorphism $\varphi$ to an automorphism of 
the trivial vector bundle. Gluing using this automorphism gives a vector bundle
$\mc F$ on $\mc C$ and an isomorphism
\begin{equation}\label{eg}
\gamma\,:\, i^*\mc F \,\longrightarrow\, E.
\end{equation}

Recall that $\bq_1$ 
denotes the projection $\Sigma\,\longrightarrow\, C$ (see \eqref{bar-q_1}). 
Similarly, we define $$\t q_1\,\,:\,\,\t \Sigma\,\longrightarrow\, \mc C$$
to be the natural projection (see \eqref{u4}). Consider the commutative square
\[\xymatrix{
\Sigma\ar[r]^{i_\Sigma}\ar[d]_{\bq_1} & \t \Sigma\ar[d]^{\t q_1}\\
C\ar[r]^{i} & \mc C
}
\]
Consider $\mc E\,:=\,\t q_1^*\mc F$ (see \eqref{eg}), and take $\theta\, :\,
i_{\Sigma}^*\mc E\, \longrightarrow\, \bq_1^*E$ to be the composite isomorphism
$$i_{\Sigma}^*\mc E\,=\,i_{\Sigma}^*\t q_1^*\mc F\,\cong\, \bq_1^*i^*\mc F
\,\xrightarrow{\,\,\bq_1^*\gamma\,\,}\,\bq_1^*E\,.$$
It is clear that the class of the quadruple $(\xi,\,i,\,\mc E,\,\theta)$ in 
${\bf Def}_{(C,\bq_1^*E)}(\C[\e])$ maps to the class of the tuple 
$(\xi,\,i)\,\in\, {\bf Def}_C(\C[\e])$. This completes the proof of the lemma.
\end{proof}

We will need to consider the deformation functor described 
in \cite[\href{https://stacks.math.columbia.edu/tag/0E3T}{Example 0E3T}]{Stk}, or
in simpler terms described in \cite[Definition 3.4.1]{Sern}. Let $f\,:\,X\,\longrightarrow\, Y$ 
be a morphism of $\C$-schemes. For $A\,\in\, {\bf Art}_\C$, let 
${\bf Def}_{X\stackrel{f}{\longrightarrow}Y}(A)$ be the collection of diagrams 
of the following type modulo equivalence:
\[\xymatrix{
	X\ar[r]\ar[d]_f & \t X\ar[d]^{\t f} \\
	Y\ar[r]\ar[d] & \t Y\ar[d] \\
	{\rm Spec}\,\C\ar[r] & {\rm Spec}\,A
}
\]
Both squares in the above diagram are Cartesian while 
the maps $\t X\,\longrightarrow\, {\rm Spec}\,A$ and $\t Y\,\longrightarrow\, {\rm Spec}\,A$
are flat. We leave it to the reader to check that 
this defines a deformation functor and that it satisfies 
the condition ({\bf H$\e$}); see \cite[\S~2.3]{Nitsure}.
The reader is referred to \cite[\href{https://stacks.math.columbia.edu/tag/06L9}{Section 06L9}]{Stk},
\cite[\href{https://stacks.math.columbia.edu/tag/0E3U}{Lemma 0E3U}]{Stk}.

\subsection{A transformation of deformation functors}

Consider the finite map $\Sigma\,\longrightarrow\, C^{(d)}$ and
the relative Quot scheme 
\begin{equation}\label{define Pi}
\mc Q'\,:=\,{\rm Quot}_{\Sigma/C^{(d)}}(\bq_1^*E,\,d)
\,\stackrel{\Pi}{\longrightarrow}\, C^{(d)}\,.
\end{equation}
We will show that there is a natural transformation of deformation 
functors
\begin{equation}\label{eta}
\eta\,:\,{\bf Def}_{(C,\bq_1^*E)}\,\longrightarrow\, {\bf Def}_{\mc Q'}\,.
\end{equation}

Let $(\xi,\,i,\,\mc E,\,\theta)$ be a quadruple
representing an equivalence class in ${\bf Def}_{(C,\bq_1^*E)}(A)$.
Consider the relative Quot scheme 
\begin{equation}
{\rm Quot}_{\t \Sigma/\mc C^{(d)}}(\mc E,\,d)\,\stackrel{\t \Pi}{\longrightarrow}\, \mc C^{(d)}
\,\longrightarrow\, {\rm Spec}\,A\,.
\end{equation}

\begin{lemma}
${\rm Quot}_{\t \Sigma/\mc C^{(d)}}(\mc E,\,d)$ is flat over $A$.
\end{lemma}

\begin{proof}
Let 
\begin{equation}\label{h1}
\bq_2\,:\,\Sigma\,\longrightarrow\, C^{(d)}\ \ \text{ and }\ \
\t q_2\,:\,\t\Sigma\,\longrightarrow\, \mc C^{(d)}
\end{equation} 
denote the restrictions of the natural projection maps
$C\times C^{(d)}\,\longrightarrow\, C^{(d)}$ and
$\mc C\times_{{\rm Spec}\,A} \mc C^{(d)}\,\longrightarrow\, \mc C^{(d)}$
respectively. Let $\mc W\,\subset\, \mc C^{(d)}$ denote 
an affine open set, and define $\t{\mc W}\,:=\,\t q_2^{-1}(\mc W)$, where
$\t q_2$ is defined in \eqref{h1}. Using the base
change property of Quot schemes we have 
$$\t \Pi^{-1}(\mc W)\,=\, {\rm Quot}_{\t{\mc W}/\mc W}(\mc E\big\vert_{\t{\mc W}},\,d)\,.$$
Define $W\,:=\,C^{(d)}\cap \mc W$ and $\t W\,:=\,\Sigma \cap \t{\mc W}$. Thus,
there is a Cartesian square 
\[\xymatrix{
\t W\ar[r]\ar[d]_{\bq_2} & \t{\mc W}\ar[d]^{\t q_2}\\
W\ar[r] & \mc W
}
\]
Both $W$ and $\t W$ are smooth affine schemes, from which it follows that 
$\mc W\,=\,W\times {\rm Spec}\,A$ and $\t{\mc W}\,=\,\t W\times {\rm Spec}\,A$.
Further, $\mc E\big\vert_{\t{\mc W}}$ is isomorphic to the pullback of 
$\bq_1^*E\big\vert_{\t W}$ to $\t W\times {\rm Spec}\,A$ along the projection map 
$\t W\times {\rm Spec}\,A \longrightarrow \t W$. Using base change property of Quot schemes we get that
$$ \t \Pi^{-1}(\mc W)\,=\, {\rm Quot}_{\t{W}/W}(\bq_1^*E\big\vert_{\t W},\,d)\times {\rm Spec}\,A\,.$$
Thus, $\t \Pi^{-1}(\mc W)$ is flat over $A$. Covering $\mc C^{(d)}$ 
by such open subsets the proof of the lemma is completed.
\end{proof}

Define 
$$
\zeta\,\,:\,\, {\rm Quot}_{\t \Sigma/\mc C^{(d)}}(\mc E,\,d)\,\longrightarrow \,{\rm Spec}\,A\,.
$$
Using the universal quotient on $\Sigma\times_{C^{(d)}}\mc Q'
\,\cong\, \t \Sigma\times_{\mc C^{(d)}}\mc Q'$ 
and the isomorphism $\theta\,:\, i_\Sigma^*\mc E\,\longrightarrow \,\bq_1^*E$ we get 
a map 
$$j\,\,:\,\,\mc Q' \,\longrightarrow\, {\rm Quot}_{\t \Sigma/\mc C^{(d)}}(\mc E,\,d)\,.$$
It is straightforward to check that there is the following Cartesian square 
\[\xymatrix{
	\mc Q'\ar[r]^<<<<<{j}\ar[d] & {\rm Quot}_{\t \Sigma/\mc C^{(d)}}(\mc E,\,d)\ar[d]^{\zeta}\\
	{\rm Spec}\,\C\ar[r] & {\rm Spec}\,A
}
\]
Now define the map $\eta$ in \eqref{eta} by the rule
\begin{equation}\label{eta-define}
\eta(\xi,\,i,\,\mc E,\,\theta)\,=\,(\zeta,\,j)\,.
\end{equation}	
It is easy to check that $\eta$ defines a natural transformation 
${\bf Def}_{(C,\bq_1^*E)}\,\longrightarrow\, {\bf Def}_{\mc Q'} $ 
and we leave this to the reader. 

\begin{remark}\label{factors through deformation of Pi}
The natural transformation $\eta$ in \eqref{eta} factors as 
$$
{\bf Def}_{(C,\bq_1^*E)}\,\longrightarrow\, {\bf Def}_{\mc Q'\,
\stackrel{\Pi}{\longrightarrow}\,C^{(d)}} \,\longrightarrow\, {\bf Def}_{\mc Q'}
$$
We leave this verification to the reader.
\end{remark}

\subsection{An open subset of $\mc Q'$}

Recall the notation of Section \ref{se9.1}. We have the universal quotient $\mc B$
on $C\times \mc Q$ (see \eqref{universal-seq-C times Q}), and the sheaf $\mc B$ is supported on 
$\mc Z\,=\, \Sigma\times_{C^{(d)}}\mc Q$. Let $p_\Sigma$ 
denote the projection from $\Sigma\times_{C^{(d)}}\mc Q$ to $\Sigma$. 
Thus, the universal quotient on $\mc B$ in \eqref{universal-seq-C times Q} factors as 
$$p_\Sigma^*\bq_1^*E\,\longrightarrow\, \mc B\,\longrightarrow\, 0\,.$$
This quotient defines a map 
\begin{equation}\label{j}
j\,\,:\,\,\mc Q\,\longrightarrow\, \mc Q'\,.
\end{equation}
Recall the Hilbert-Chow morphism $\phi\,:\,\mc Q\,\longrightarrow\, C^{(d)}$ in \eqref{hc}.
For a closed point $q\,\in\, \mc Q$, let $E\,\stackrel{q}{\longrightarrow}\, \mc B_q$
denote the quotient on $C$ corresponding to it. 
The map $j$ in \eqref{j} has the description 
\[(E\,\stackrel{q}{\longrightarrow}\, \mc B_q) \,\longmapsto\, 
\left(E\big\vert_{\phi(q)}\,\stackrel{j(q)}{\longrightarrow}\, \mc B_q\right)\,.\]
It follows easily that $j$ is an inclusion on closed points.
In fact, it can be easily checked that $j$ induces an
inclusion $\mc Q(T)\,\longrightarrow\, \mc Q'(T)$ for any $\C$-scheme $T$. In
particular, it induces an inclusion at the level of Zariski tangent
spaces.

Let
\begin{equation}\label{ab}
\mc Q^0\,\subset\, \mc Q
\end{equation}
be the open subset parametrizing the 
quotients $\{q\,:\,E\,\longrightarrow\, \mc B_q\,\}$ such that
$\mc B_q\,\cong\, \mc O_D$, where $D\,=\,\phi(q)$. 
A proof that $\mc Q^0$ is open can be found in 
\cite[Definition 24, Lemma 25]{GS-nef}.
Let
\begin{equation}\label{define j_0}
j_0\,\,:\,\,\mc Q^0\,\subset\,\mc Q\,\longrightarrow\, \mc Q'
\end{equation}
be the restriction of $j$ in \eqref{j} to the open subset $\mc Q^0$.

We will use the following general fact in the proof of Lemma \ref{lem-1}.

\begin{lemma}\label{general fact open immersion}
Let $f\,:\,Y\,\longrightarrow\, X$
be a morphism of $\C$-schemes of finite type. 
Assume that $Y$ is smooth, irreducible and
$f$ is injective on closed points of $Y$. For all closed points $y\in Y$ 
assume that the map $df$ of Zariski tangent spaces
$$df_y\,:\,T_yY\,\longrightarrow \,T_{f(y)}X$$
is an isomorphism. Then the image $f(Y)$ is an open
subset of $X$, and the map $f$ is an isomorphism between $Y$ and $f(Y)$.
\end{lemma}

\begin{proof}
Let $X'$ denote an irreducible component of $X$ 
which contains the image $f(Y)$ with the reduced sub-scheme structure. 
As $f$ is injective 
on closed points, it follows that ${\rm dim}(Y)\,\leqslant\, {\rm dim}(X')$. 
We also have
$${\rm dim}_\C\,T_yY\,=\,{\rm dim}(Y)\,\leqslant\, {\rm dim}(X')\,\leqslant\, 
{\rm dim}_\C\,T_{f(y)}X'\,\leqslant \,{\rm dim}_\C\,T_{f(y)}X\,.$$
The hypothesis of the Lemma says that ${\rm dim}_\C\,T_yY\,=\,{\rm dim}_\C\,T_{f(y)}X$.
It follows that all the quantities appearing in the above equation
are equal. In particular, we get that $X'$ is smooth at $f(y)$. This shows
that there are equalities at each step in the following sequence of 
inequalities
$${\rm dim}_\C\,T_{f(y)}X'\,=\,{\rm dim}(X')\,\leqslant \,{\rm dim}_{f(y)}(X)
\,\leqslant\, {\rm dim}_\C\,T_{f(y)}X\,.$$
This proves that $X$ is smooth at $f(y)$ of same dimension as $Y$.
Now the lemma follows easily.
\end{proof}
 
\begin{lemma}\label{lem-1}
The map $j_0$ in \eqref{define j_0} is an open immersion.
\end{lemma}

\begin{proof}
Lemma \ref{general fact open immersion} will be applied to $j_0$. 
We know that $\mc Q^0$ in \eqref{ab} is smooth, irreducible 
and $j_0$ is an inclusion on closed points and tangent vectors. 

Fix a point $q\,\in\, \mc Q^0$; let $q'\,:=\,j_0(q)\,\in\, \mc Q'$ and 
$D\,:=\,\phi(q)$. Then the point $q'$ corresponds to a short exact sequence on $D$
$$
0\,\longrightarrow\, A'\,\longrightarrow\, E\big\vert_D
\,\stackrel{q'}{\longrightarrow}\, \mc O_D\,\longrightarrow\, 0\,.$$
It is clear that $A'\,\cong\, \mc O_D^{\oplus (r-1)}$.
As ${\rm Ext}^1_{\mc O_D}(A',\,\mc O_D)\,=\,0$, an easy computation using 
the short exact sequence in \cite[Proposition 2.2.7]{HL} shows that 
$$\dim_\C T_{q}(\mc Q^0)\,=\,\dim_\C T_{q'}(\mc Q')\,.$$
It follows that $j_0$ induces an isomorphism of Zariski tangent spaces.
By Lemma \ref{general fact open immersion}, $j_0$ is an open immersion.
\end{proof}

\begin{lemma}\label{codim Q^0 Q}
Let $\mc Z_1\,:=\,\mc Q\setminus \mc Q^0$ be the complement (see
\eqref{ab}). Then ${\rm codim}(\mc Z_1,\,\mc Q)\,\geqslant\, 3$.
\end{lemma}

\begin{proof}
Denote by $$\Delta\,\subset\, C^{(d)}$$ the images of all points
$(c_1,\,c_2,\,\cdots,\,c_d)\,\in \,C^{d}$ for which $c_i\,=\,c_j$ 
for some $1\,\leqslant\, i\,<\,j\,\leqslant\, d$.
For the Hilbert-Chow morphism $\phi$ in \eqref{hc} it is evident that $\phi(\mc Z_1)\,\subset\, \Delta$.
For a point $D\,\in \,\Delta$, let $\mc Q_D$ denote the fiber of $\phi$ over $D$. 
It follows from \cite[Corollary 5.6]{GS} that the inequality
${\rm codim}(\mc Z_1\cap \mc Q_D,\,\mc Q_D)\,\geqslant\, 2$ holds.
In \cite{GS} the open subset $\mc Q_D^0$ is denoted by $V$; see 
just before \cite[Lemma 5.1]{GS}. As ${\rm codim}(\Delta,\,C^{(d)})\,=\,1$,
the lemma follows.
\end{proof}

\begin{lemma}\label{H^1 large open set}
Take $X$ to be smooth, and let $f\,:\,U\, \hookrightarrow\, X$ be an open immersion. Assume
that the inequality ${\rm codim}(X\setminus U,\,X)\,\geqslant\, 3$ holds. For a coherent sheaf 
$F$ on $X$, denote by $F_U$ its restriction to $U$. For a locally free sheaf $F$ on $X$,
the following two hold:
\begin{enumerate}
\item $R^1f_*(F_U)\,=\,0$, and

\item $H^1(X,\,F)\,\cong\, H^1(U,\,F_U)$. 
\end{enumerate}
\end{lemma}

\begin{proof}
Define $Z\,:=\,X\setminus U$, and let $p\,\in \,Z$ be a closed point. 
Let $X_p\,:=\,{\rm Spec}(\mc O_{X,p})$, $U_p\,:=\,U\bigcap X_p$ and $Z_p\,:=\,X_p\setminus U_p$ 
with the reduced scheme structure. To prove the first assertion, 
it suffices to show that $$R^1f_*(F_U)\big\vert_{X_p}\,=\,H^1(U_p,\,\mc O_{U_p})\,=\,0.$$ 
For this, using \cite[Chapter 3, Ex 2.3]{Ha},
it is enough to show that $H^i_{Z_p}(X_p,\,\mc O_{X_p})\,=\,0$ for $i\,=\,1,\,2$.
Now this follows using \cite[Chapter 3, Ex 3.4]{Ha} and
\cite[Definition 1.2.6, Theorem 2.1.2(b)]{Bruns-Herzog}.
	
The second assertion follows from the first one and the Leray spectral sequence.
\end{proof}

\subsection{A commutative square}

In this subsection we want to establish the 
commutativity (up to sign) 
of the diagram in \eqref{main deformation diagram}. Consider the following
square which, a priori, need not be commutative:
\begin{equation}\label{main deformation diagram}
\begin{tikzcd}
\mc D_{\bq_1^*E}(\C[\e])\ar[r]\ar[d, "\alpha", labels=left]
\arrow[rrr, bend left=10, "\gamma", labels=above right]& 
{\bf Def}_{(C,\bq_1^*E)}(\C[\e])\ar[r]&
{\bf Def}_{\mc Q'}(\C[\e])\ar[r] & {\bf Def}_{\mc Q^0}(\C[\e])\ar[d, "\delta"]\\
H^1(\Sigma,\,ad(\bq_1^*E))\ar[rrr, "\beta", labels=below]&&&
H^1(\mc Q^0,\,T_{\mc Q^0})
\end{tikzcd}
\end{equation}
The two vertical arrows in \eqref{main deformation diagram} have the usual description 
using cocycles. The map $\gamma$ is the composition of the maps 
which arise from the natural transformations defined above. 
The map $\beta$ in \eqref{main deformation diagram} is the following composition of maps: 
\begin{align}\label{seq of maps}
H^1(\Sigma,\,ad(\bq_1^*E))&\,\stackrel{\Xi}{\longrightarrow}\, H^1(\Sigma,\, \Phi_*\ms Hom(\mc A,\mc B))\\
&\stackrel{\sim}{\longrightarrow}\, H^1(\mc Z,\,\ms Hom(\mc A,\mc B))\nonumber\\
&\stackrel{\sim}{\longrightarrow}\, H^1(\mc Q,\,T_{\mc Q})\nonumber\\
&\stackrel{\sim}{\longrightarrow}\, H^1(\mc Q^0,\,T_{\mc Q^0}).\nonumber
\end{align}
The map $\Xi$ in \eqref{seq of maps} is the one in Theorem \ref{preliminiary ses}. In \eqref{seq of maps},
the second map is an isomorphism, see Corollary \ref{corollary main theorems}. The third isomorphism follows 
using Proposition \ref{T_Q} and the fact that
$\mc Z\,\longrightarrow \,\mc Q$ (see \eqref{dpi}) is a finite map.
The last isomorphism in \eqref{seq of maps} follows 
using Lemma \ref{codim Q^0 Q} and Lemma \ref{H^1 large open set}.

\begin{proposition}\label{prop-s}
The diagram \eqref{main deformation diagram} is commutative up
to a minus sign. 
\end{proposition}

\begin{proof}
We will start with a cocycle description 
of any element of $\mc D_{\bq_1^*E}(\C[\e])$ and show that the 
resulting element in $H^1(\mc Q^0,\,T_{\mc Q^0})$ is independent of whether
we take the path $\beta\circ \alpha$ or $\delta\circ \gamma$.
Recall the following commutative diagram from \eqref{dpi} (see Section \ref{se9.1} 
for the definition of the maps):
$$
	\begin{tikzcd}
		C\ar[d,equal] & \mc Z\ar[r]\ar[d]\arrow[rr, bend left=15, "\overline{p}_2", labels=above right]
		\arrow[l, "\overline{p}_1", labels=above] & 
		C\times \mc Q\ar[r, "p_2", labels=below]\ar[d, "\Phi"] & \mc Q\ar[d, "\phi"] \\
		C & \Sigma \ar[r]\arrow[rr, bend right=15, "\overline{q}_2", labels=below right]
			\arrow[l, "\overline{q}_1", labels=above] & 
C\times C^{(d)}\ar[r, "q_2", labels=above] & C^{(d)}
\end{tikzcd}
$$

Let 
\begin{equation}\label{affine open cover}
C^{(d)}\,=\,\bigcup_{i=1}^nU_i
\end{equation}
be an affine open cover for $C^{(d)}$. Set
\begin{equation}\label{o1}
\Sigma_i\,:=\,\bq_2^{-1}(U_i),\ \ \mc Q_{i}\,:=\,\phi^{-1}(U_i)\ \ \text{ and }\ \ 
\mc Z_i\,:=\,\overline{p}_2^{-1}(\mc Q_i)\, ;
\end{equation}
see \eqref{hc}, \eqref{h1} and \eqref{dpi} for $\phi$, $\bq_2$ and $\overline{p}_2$
respectively. Let $U_{ij}\,=\,U_i\cap U_j$;
similarly define the open subsets 
\begin{equation}\label{o2}
\Sigma_{ij}\,:=\,\bq_2^{-1}(U_{ij}),\ \ \mc Q_{ij}\,:=\,\phi^{-1}(U_{ij})\ \ \text{ and }\ \ 
\mc Z_{ij}\,:=\,\overline{p}_2^{-1}(\mc Q_{ij}).
\end{equation}

An element of $\mc D_{\bq_1^*E}(\C[\e])$ is a locally free sheaf $\mc E$
on $\Sigma[\e]$ together with an isomorphism 
$$\theta\,:\,\bq_1^*E\,\longrightarrow\, \mc E\big\vert_\Sigma\,.$$
As the deformation is trivial on $\Sigma_i[\e]$, it follows that 
there is an isomorphism 
$$\varphi_i\,\,:\,\,\left(\bq_1^*E\big\vert_{\Sigma_i}\right)[\e]
\,\longrightarrow\, \mc E\big\vert_{\Sigma_i[\e]}$$
such that the following diagram, in which the lower row is obtained
by pulling back $\varphi_i$ along the closed immersion $\Sigma_i\longrightarrow \Sigma_i[\e]$,
\begin{equation}\label{proof diagram commutes e1}
\begin{tikzcd}
\bq_1^*E\big\vert_{\Sigma_i} \ar[r, "\theta\vert_{\Sigma_i}"]\ar[d,equal] & 
\mc E\big\vert_{\Sigma_i}\ar[d,equal]\\
\bq_1^*E\big\vert_{\Sigma_i} \arrow[r, "\varphi_i\vert_{\Sigma_i}"] & \mc E\big\vert_{\Sigma_i}
\end{tikzcd}
\end{equation}
commutes. Consider the following composite map
\[
\left(\bq_1^*E\big\vert_{\Sigma_{ij}}\right)[\e]\,\, \xrightarrow{\,\,\varphi_{i}\vert_{\Sigma_{ij}}\,\,}
\,\, \mc E\big\vert_{\Sigma_{ij}[\e]}\,\,\xrightarrow{\,\,\varphi_{j}^{-1}\vert_{\Sigma_{ij}}\,\,}
\,\, \left(\bq_1^*E\big\vert_{\Sigma_{ij}}\right)[\e]
\]
(see \eqref{o2}).
In view of \eqref{proof diagram commutes e1} it follows that the 
restriction of this map to $\Sigma_{ij}$ is the identity.
Thus, we may write
\begin{equation}\label{tilde psi}
\t\psi_{ij}\,:=\,\varphi_{j}^{-1}\big\vert_{\Sigma_{ij}}\circ \varphi_{i}\vert_{\Sigma_{ij}}
\,=\,{\rm Id}+\e \psi_{ij},
	\end{equation} 
where $\psi_{ij}\,\in\, \Gamma(\Sigma_{ij},\,\ms End(\bq_1^*E))$.
These $\{\psi_{ij}\}_{ij}$ define a cohomology class in 
$$[\mc E]\,\in\, H^1(\Sigma, \,\ms End(\bq_1^*E))\,.$$
The image of this class $[\mc E]$ in $H^1(\Sigma,\, ad(\bq_1^*E))$
is precisely the image of the deformation $\mc E$
under the map $\alpha$ in diagram \eqref{main deformation diagram}.
	
Consider the universal quotient \eqref{universal-seq-C times Q}
on $C\times \mc Q$. Restricting this to $\mc Z$, and using $\mc B$
is supported on $\mc Z$, we get a 
quotient which we denote by 
\begin{equation}
	u\,\,:\,\, \overline{p}_1^*E\,\longrightarrow\,\mc B\,.
\end{equation}
Recall the two commutative diagrams 
\eqref{preliminary ses e2} and \eqref{preliminary ses e3}.
Using these, together with the projection formula
and Corollary \ref{corollary main theorems},
it is easily checked that there is a commutative diagram
\[
\xymatrix{
\Gamma(\Sigma_{ij},\,\ms End(\bq_1^*E))\ar[r]\ar@{=}[d] & 
\Gamma(\Sigma_{ij},\, ad(\bq_1^*E))\ar@{=}[d]\ar[r]^<<<<{\Xi} & 
\Gamma(\Sigma_{ij},\, \Phi_*\ms Hom(\mc A,\mc B))\ar@{=}[d]\\
\Gamma(\mc Z_{ij},\,\ms End(\overline{p}_1^*E))\ar[r] & 
\Gamma(\mc Z_{ij},\,ad(\overline{p}_1^*E))\ar[r] &
\Gamma(\mc Z_{ij},\, \ms Hom(\mc A,\mc B))
}
\]
(see \eqref{o2}).
	
Consider $\Phi^*\psi_{ij}\,\in \,\Gamma(\mc Z_{ij},\,\ms End(\overline{p}_1^*E))$
and its image
\begin{equation}\label{et}
\Theta_{ij}\,\in\, \Gamma(\mc Z_{ij},\,\ms Hom(\mc A,\mc B))
\end{equation}
under the map $\ms End(\overline{p}_1^*E)\,\longrightarrow\, \ms Hom(\mc A,\,\mc B)$.
In order to avoid the notation from becoming too cumbersome, we
continue to denote the restriction of a sheaf $F$ to an open set $U$ by $F$,
instead of $F\vert_U$. Using \eqref{preliminary ses e2} one checks that the map $\Theta_{ij}$ is the composite
$$\mc A\,\longrightarrow\, \mc A\vert_{\mc Z}\,\longrightarrow\, 
\overline{p}_1^*E\, \xrightarrow{\,\,\Phi^*\psi_{ij}\,\,}\,
\overline{p}_1^*E\,\stackrel{u}{\longrightarrow}\, \mc B\, ,$$
where $\overline{p}_1$ is defined in \eqref{dpi}.
Let $\mc A'$ denote the following kernel of the map of sheaves $u$ on $\mc Z$
\begin{equation}\label{relative universal exact sequence}
0\,\longrightarrow\, \mc A'\,\stackrel{l}{\longrightarrow}\,
\overline{p}_1^*E\,\stackrel{u}{\longrightarrow} \,\mc B\,\longrightarrow\, 0\,.
\end{equation}
Restriction of the universal exact sequence \eqref{universal-seq-C times Q}
to $\mc Z$ produces the commutative diagram
\[
\xymatrix{
& \mc A\vert_{\mc Z}\ar[r]\ar[d] &p_1^*E\big\vert_{\mc Z} \ar[r]\ar@{=}[d]& \mc B\ar[r]\ar@{=}[d] &0\\
0\ar[r] & \mc A'\ar[r]^l &\overline{p}_1^*E \ar[r]^u& \mc B\ar[r] &0
	}
	\]
(the map $p_1$ is defined in \eqref{t5}).
It follows that the map $\Theta_{ij}$ in \eqref{et} factors uniquely 
through $\Theta^0_{ij}\,\in\, \Gamma(\mc Z_{ij},\,\ms Hom(\mc A',\mc B))$
\begin{equation}\label{define Theta^0}
\begin{tikzcd}
\mc A\ar[r]\arrow[rrrrr, bend left=15, "\Theta_{ij}", labels=above right]& 
\mc A\vert_{\mc Z}\ar[r] & \mc A' \ar[r,"l"]
\arrow[rrr, bend right=15, "\Theta^0_{ij}", labels=below right]&
\overline{p}_1^*E\arrow[r,"\Phi^*\psi_{ij}"] &
\overline{p}_1^*E\ar[r,"u"] & \mc B\,.
\end{tikzcd}
\end{equation}

We now summarize what has been done so far. Given a deformation 
$\mc E\,\in \,\mc D_{\bq_1^*E}(\C[\e])$ we have defined, using cocycles, 
a cohomology class $[\mc E]\,\in\, H^1(\Sigma,\, \ms End(\bq_1^*E))$. 
The class $\beta\circ\alpha(\mc E)$ 
in diagram \eqref{main deformation diagram} is the image of $[\mc E]$ 
along the maps 
\[
\begin{array}{ccccc}
\stackrel{\{\psi_{ij}\}}{\rotatebox[origin=c]{-90}{$\in$}} & 
\stackrel{\{\Theta^0_{ij}\}}{\rotatebox[origin=c]{-90}{$\in$}} & 
\stackrel{\{\Theta_{ij}\}}{\rotatebox[origin=c]{-90}{$\in$}}\\
H^1(\Sigma,\, \ms End(\bq_1^*E))\,\longrightarrow & H^1(\mc Z,\,\ms Hom(\mc A',\mc B)) 
\,\longrightarrow & H^1(\mc Z,\,\ms Hom(\mc A,\,\mc B))\\ 
& &\\
& & \qquad \longrightarrow H^1(\mc Q^0,\,T_{\mc Q^0})
\end{array}
\]
The last map is the composite of the last two maps in \eqref{seq of maps}.

Next we analyze the map $\delta\circ \gamma$ in 
\eqref{main deformation diagram}. Recall the open covers
from \eqref{affine open cover} and \eqref{o1}.
The deformation $\mc E$ defines the deformation 
${\rm Quot}_{\Sigma[\e]/(C^{(d)}[\e])}(\mc E,\,d)$
of $\mc Q'$; see \eqref{eta}, \eqref{eta-define}. 
This in turn defines a deformation of every open
subset of $\mc Q'$, and in particular, a deformation of 
$\mc Q^0$. This is precisely the deformation corresponding
to the element $\gamma(\mc E)\,\in\,{\bf Def}_{\mc Q^0}(\C[\e])$. 
We denote this open sub-scheme by 
\begin{equation}
	\t{\mc Q}^0\,\subset\, {\rm Quot}_{\Sigma[\e]/(C^{(d)}[\e])}(\mc E,\,d)\,.
\end{equation}
The underlying set of points of $\t{\mc Q}^0$ 
is the same as those of $\mc Q^0$.
The space $\t{\mc Q}^0$ defines a cocycle as follows.
Set
$$\mc Q^0_i\,:=\,\mc Q_i\cap \mc Q^0\,\ \ \text{ and }\,\ \ \mc Q^0_{ij}\,:=\,\mc Q_{ij}\cap \mc Q^0\, ,$$
where $\mc Q_i$ and $\mc Q_{ij}$ are defined in \eqref{o1} and \eqref{o2} respectively. 
Using the base change property, and the fact that $\mc E$ restricted to 
$U_i[\e]$ is trivial, it follows immediately that 
$$\t{\mc Q}^0\vert_{U_i[\e]}\,=\,\mc Q^0_i[\e]\,.$$
The isomorphism $\t\psi_{ij}$ (see \eqref{tilde psi}) defines an isomorphism 
\begin{equation}\label{o3}
\Psi_{ij}\,\,:\,\,\mc Q^0_{ij}[\e]\,\longrightarrow\,\mc Q^0_{ij}[\e]\, ,
\end{equation}
which gives rise to an element of $\Gamma(\mc Q^0_{ij},\, T_{\mc Q^0})$.
We will compute this element of $\Gamma(\mc Q^0_{ij},\, T_{\mc Q^0})$.

Define $\mc Z^0\,:=\, \Sigma\times_{C^{(d)}}\mc Q^0$; 
recall that this is an open subset of both $\mc Z\,=\,\Sigma\times_{C^{(d)}}\mc Q$
and $\Sigma\times_{C^{(d)}}\mc Q'$.
Observe that the restriction of \eqref{relative universal exact sequence}
to $\mc Z^0$ is the restriction, to $\Sigma\times_{C^{(d)}}\mc Q^0$, of the universal
quotient on $\Sigma\times_{C^{(d)}}\mc Q'$.
As before, we have open subsets $\mc Z^0_{i}\,:=\, \mc Z^0
\bigcap \mc Z_{i}$ and $\mc Z^0_{ij}\,:=\, \mc Z^0 \bigcap \mc Z_{ij}$ of $\mc Z^0$
(see \eqref{o1} and \eqref{o2}).
The isomorphism $\Psi_{ij}$ in \eqref{o3} is defined by the lower
row in the following commutative diagram of sheaves on $Z^0_{ij}[\e]$
	\[
	\begin{tikzcd}
	0\ar[r] & \mc A'[\e]\ar[r, "l\lbrack\e\rbrack"]\ar[d,equal] & 
	\overline{p}_1^*E[\e]\arrow[rr,"u\lbrack\e\rbrack"] & &
		\mc B[\e]\ar[r]\ar[d,equal] &0\,\phantom{.}\\
		0\ar[r] & \mc A'[\e]\arrow[r,"l'"] & 
		\overline{p}_1^*E[\e]\arrow[rr, "u\lbrack\e\rbrack\circ \Phi^*\t\psi_{ij}"]
			\arrow[u, "\Phi^*\t\psi_{ij}"] & &
		\mc B[\e]\arrow[r] &0\,.
	\end{tikzcd}
	\]
	Let $a,\,b\,\in\,\overline{p}_1^*E$. 
	If we write an element 
	$$a\oplus b\otimes \e\,\in\, \overline{p}_1^*E[\e]\,=\,\overline{p}_1^*E \oplus
\overline{p}_1^*E\otimes \e$$ 
	as the column vector 
	$[a,\,b\otimes\e]^t$, then the matrix of $\Phi^*\t\psi_{ij}$ has the expression
	\[
	\left(\begin{array}{c c}
	{\rm Id} & 0\\
		\Phi^*\psi_{ij} & {\rm Id}
	\end{array}\right)\,.
	\]
	This shows that the map $l'$ is equal to 
	\[
	\left(\begin{array}{c c}
		l & 0\\
		-\Phi^*\psi_{ij} & l
	\end{array}\right)\,.
	\]
This gives us the section 
$$-u\circ \Phi^*\psi_{ij}\circ l\,\in \,\Gamma(\mc Z^0_{ij},\, \ms Hom(\mc A',\,\mc B))\,.$$
The images of these in $\Gamma(\mc Q^0_{ij},\,T_{\mc Q^0})$ along the map 
$$\Gamma(\mc Z^0_{ij}, \,\ms Hom(\mc A',\,\mc B))\,\longrightarrow \,\Gamma(\mc Z^0_{ij},
\, \ms Hom(\mc A,\,\mc B))
\,\longrightarrow \,\Gamma(\mc Q^0_{ij}, \,T_{\mc Q^0})$$
define the required cocycle
which represents the class $\delta\circ\gamma(\mc E)$
(see \eqref{main deformation diagram} for $\delta$ and $\gamma$). It is evident using 
\eqref{define Theta^0} that
$$-u\circ \Phi^*\psi_{ij}\circ l\,=\,-\Theta^0_{ij}\,.$$
This completes the proof of the proposition.
\end{proof}

Recall that $g_C$ denotes the genus of the curve $C$.

\begin{theorem}\label{th}
Let $d,\,g_C\,\geqslant\, 2$.
The composite map $${\bf Def}_{(C,\bq_1^*E)}(\C[\e])\,\longrightarrow\,
{\bf Def}_{\mc Q^0}(\C[\e])\,\stackrel{\delta}{\longrightarrow}\,
H^1(\mc Q^0,\,T_{\mc Q^0})$$ in \eqref{main deformation diagram}
is surjective.
\end{theorem}

\begin{proof}
Denote the composite map in the statement of the theorem by 
$\delta_0$. From Proposition \ref{prop-s} we conclude that the square in the following
diagram commutes up to a minus sign
\begin{equation}\label{deformation diagram ses}
\begin{tikzcd}
\mc D_{\bq_1^*E}(\C[\e])\ar[r,"\gamma_0"]\ar[d, "\alpha", labels=left]& 
{\bf Def}_{(C,\bq_1^*E)}(\C[\e])\ar[d, "\delta_0"]\ar[r] &{\bf Def}_C(\C[\e])\\
H^1(\Sigma,\,ad(\bq_1^*E))\ar[r, "\beta", labels=below]& H^1(\mc Q^0,\,T_{\mc Q^0})
\ar[r] & H^1(C,\,T_C)\,.
\end{tikzcd}
\end{equation}
The two rows in the above diagram are short exact sequences;
see Lemma \ref{ses deformation functors} and Theorem \ref{cor-cohomology of T_Q-1}.
It follows that there is an induced map
\begin{equation}\label{define f final theorem}
f\,\,:\,\,{\bf Def}_C(\C[\e])\,\longrightarrow\, H^1(C,\,T_C)\,
\end{equation}
which makes the right-side of \eqref{deformation diagram ses} a commutative square.
Although there is an isomorphism of vector spaces between 
the domain and codomain of $f$, it is not clear if $f$ itself
is an isomorphism. We want to show that $f$ is an isomorphism.

We will show that every element in the kernel of 
the above defined map $\delta_0$ is in $\mc D_{\bq_1^*E}(\C[\e])$.
Take an element in the kernel of $\delta_0$.
For ease of notation we denote it by a pair $(\mc C,\,\mc E)$.
Recall the definitions of the maps $\Pi$ and $j_0$ 
(see \eqref{define Pi}, \eqref{define j_0}). Let
$$\phi_0\,:\, \mc Q^0 \,\longrightarrow\, C^{(d)}$$
be the restriction of the Hilbert-Chow map $\phi$ in \eqref{hc}.
Observe that $\phi_0\,=\,\Pi\circ j_0$.
Using Remark \ref{factors through deformation of Pi} we conclude 
that there is a commutative diagram 
\[
\xymatrix{
{\bf Def}_{(C,\bq_1^*E)}\ar[r] & {\bf Def}_{\mc Q'\stackrel{\Pi}{\longrightarrow}C^{(d)}}
\ar[r]\ar[d] & {\bf Def}_{\mc Q'}\ar[d]\\
&{\bf Def}_{\mc Q^0\stackrel{\phi_0}{\longrightarrow}C^{(d)}}
\ar[r] & {\bf Def}_{\mc Q^0}
}
\]
Applying the Grothendieck spectral sequence to 
the composite $\phi_0\,=\, \phi\circ j_0$, and using Lemma \ref{H^1 large open set}
and Corollary \ref{cor-direct image of structure sheaf},
we conclude that $R^1\phi_{0*}(\mc O_{\mc Q^0})\,=\,0$ and 
$\phi_{0*}\mc O_{\mc Q^0}\,=\,\mc O_{C^{(d)}}$.
Using \cite[\href{https://stacks.math.columbia.edu/tag/0E3X}{Lemma 0E3X}]{Stk}
it follows that if a deformation of $\phi_0$ induces
the trivial deformation of $\mc Q^0$, then the deformation 
itself was the trivial one. In particular, the deformation 
of the base is also the trivial one, that is, $\mc C^{(d)}$
is the trivial deformation of $C^{(d)}$. Using \cite[Lemma 4.3]{Kempf}
it follows that $\mc C$
is the trivial deformation. Consequently, the pair $(\mc C,\mc E)$
is in $\mc D_{\bq_1^*E}(\C[\e])$.

Recall the map $f$ from \eqref{define f final theorem}.
As ${\rm kernel}(\alpha)\,=\,{\rm kernel}(\delta_0)$ and 
${\rm cokernel}(\alpha)\,=\,0$, using Snake lemma it follows 
that ${\rm kernel}(f)\,=\,0$. Considering the dimensions it is deduced
that $f$ is surjective. It now follows that $\delta_0$ is surjective. 
\end{proof}

Lemma \ref{codim Q^0 Q} and Lemma \ref{H^1 large open set}
show that there is an isomorphism $H^1(\mc Q,\,T_{\mc Q})\,
\stackrel{\sim}{\longrightarrow}\, H^1(\mc Q^0,\,T_{\mc Q^0})$.
In other words, an infinitesimal deformation of $\mc Q$ is completely 
determined by its restriction to $\mc Q^0$. 
Theorem \ref{th} shows that all first order deformations of $\mc Q$ arise 
from elements of $${\bf Def}_{(C,\bq_1^*E)}(\C[\e])$$ as described above.
Recall the following commutative diagram from \eqref{deformation diagram ses},
in which the two rows are short exact sequences. 
\begin{equation*}
	\begin{tikzcd}
		\mc D_{\bq_1^*E}(\C[\e])\ar[r,"\gamma_0"]\ar[d, "\alpha", labels=left]& 
		{\bf Def}_{(C,\bq_1^*E)}(\C[\e])\ar[d, "\delta_0"]\ar[r] &{\bf Def}_C(\C[\e])\ar[d, "f"]\\
		H^1(\Sigma,\,ad(\bq_1^*E))\ar[r, "\beta", labels=below]& H^1(\mc Q^0,\,T_{\mc Q^0})
		\ar[r] & H^1(C,\,T_C)\,.
	\end{tikzcd}
\end{equation*}
We proved that $f$ is an isomorphism. In particular, it follows that ${\rm Ker}(\delta_0)={\rm Ker}(\alpha)$. 
In view of the isomorphism $\mc D_{\bq_1^*E}(\C[\e])\,\stackrel{\sim}{\longrightarrow}
\,H^1(\Sigma,\,{\rm End}(\bq_1^*E))$, 
it follows that 
$${\rm Ker}(\delta_0)\,=\,{\rm Ker}(\alpha)\,=\,H^1(\Sigma,\,\mc O_\Sigma)
\,=\,H^1(C,\,\mc O_C)\oplus H^1(C^{(d-1)},\,\mc O_{C^{(d-1)}}).$$
For the last equality we have used the K\"unneth formula and the fact that $\Sigma\cong C\times C^{(d)}$.
Since we are assuming that $g_C\geqslant 2$, it follows that ${\rm Ker}(\delta_0)\neq0$.

Another way to see that we have a non-trivial kernel is to observe the following.
Note that elements of
$H^1(\Sigma,\mc O_{\Sigma})$ correspond to first order deformations of the line bundle $\mc O_{\Sigma}$ on $\Sigma$. Given an element $\alpha\in H^1(\Sigma,\mc O_{\Sigma})$ let us denote the corresponding line bundle on $\Sigma[\epsilon]$ by $L_{\alpha}$. It can be checked easily that under the inclusion
$$H^1(\Sigma, \mc O_{\Sigma})\subset H^1(\Sigma, \ms End (\overline{q}_1^*E))
\,=\, \mc D_{\overline{q}_1^*E}(\mb C[\epsilon])$$
the image of the element $\alpha$ corresponds to the first order deformation $L_{\alpha}\otimes
(\overline{q}_1^*E[\epsilon])$ of the bundle $\overline{q}_1^*E$. Now by definition (see \ref{eta}), under the map
\begin{equation}
\mc D_{\overline{q}_1^*E}(\mb C[\epsilon]) \subset \,{\bf Def}_{(C,\bq_1^*E)}(\mb C[\epsilon])\,\longrightarrow \, {\bf Def}_{\mc Q'}(\mb C[\epsilon])\,.
\end{equation}
where $\mc Q'\,=\,{\rm Quot}_{\Sigma/C^{(d)}}(\overline{q}_1^*E, d)$ this deformation
$L_{\alpha}\otimes (\overline{q}_1^*E[\epsilon])$ of the bundle $\overline{q}_1^*E$ maps to the deformation 
$${\rm Quot}_{\Sigma[\epsilon]/C[\epsilon]^{(d)}}(L_{\alpha}\otimes (\overline{q}_1^*E[\epsilon]), \, d)$$
of $\mc Q'$. But we have an isomorphism
$${\rm Quot}_{\Sigma[\epsilon]/C[\epsilon]^{(d)}}( (\overline{q}_1^*E[\epsilon]), \, d)
\cong
{\rm Quot}_{\Sigma[\epsilon]/C[\epsilon]^{(d)}}(L_{\alpha}\otimes (\overline{q}_1^*E[\epsilon]), \, d)$$
which is produced by twisting the universal quotient by the line bundle $L_{\alpha}$. From the base change property of Quot schemes it follows that
$${\rm Quot}_{\Sigma[\epsilon]/C[\epsilon]^{(d)}}( (\overline{q}_1^*E[\epsilon]), \, d)\cong \mc Q'[\epsilon]\,,$$
so it is the trivial deformation. Hence the element $\alpha$ induces the trivial deformation of $\mc Q$.

\section*{Acknowledgements}

We thank the referee for helpful comments to improve the exposition.

The first-named
author is partially supported by a J. C. Bose Fellowship (JBR/2023/000003).

\section*{Funding and Conflicts of interests/Competing interests}

The authors have no relevant financial or non-financial interests to disclose.
The authors have no competing interests to declare that are relevant to the content of this article.

\newcommand{\etalchar}[1]{$^{#1}$}

\end{document}